\documentclass{amsart}

\usepackage{amsthm, amsfonts, amssymb, amsmath, latexsym, enumerate, times}
\usepackage[latin1]{inputenc}
\usepackage{mathrsfs}
\usepackage{mathtools}
\usepackage{letltxmacro}
\LetLtxMacro\orgvdots\vdots
\LetLtxMacro\orgddots\dots

\makeatletter
\DeclareRobustCommand\vdots{%
	\mathpalette\@vdots{}%
}
\newcommand*{\@vdots}[2]{%
	\sbox0{$#1\cdotp\cdotp\cdotp\m@th$}%
	\sbox2{$#1.\m@th$}%
	\vbox{%
		\dimen@=\wd0 %
		\advance\dimen@ -3\ht2 %
		\kern.5\dimen@
		\dimen@=\wd2 %
		\advance\dimen@ -\ht2 %
		\dimen2=\wd0 %
		\advance\dimen2 -\dimen@
		\vbox to \dimen2{%
			\offinterlineskip
			\copy2 \vfill\copy2 \vfill\copy2 %
		}%
	}%
}
\DeclareRobustCommand\ddots{%
	\mathinner{%
		\mathpalette\@ddots{}%
		\mkern\thinmuskip
	}%
}
\newcommand*{\@ddots}[2]{%
	\sbox0{$#1\cdotp\cdotp\cdotp\m@th$}%
	\sbox2{$#1.\m@th$}%
	\vbox{%
		\dimen@=\wd0 %
		\advance\dimen@ -3\ht2 %
		\kern.5\dimen@
		\dimen@=\wd2 %
		\advance\dimen@ -\ht2 %
		\dimen2=\wd0 %
		\advance\dimen2 -\dimen@
		\vbox to \dimen2{%
			\offinterlineskip
			\hbox{$#1\mathpunct{.}\m@th$}%
			\vfill
			\hbox{$#1\mathpunct{\kern\wd2}\mathpunct{.}\m@th$}%
			\villi
			\hbox{$#1\mathpunct{\kern\wd2}\mathpunct{\kern\wd2}\mathpunct{.}\m@th$}%
		}%
	}%
}
\makeatother

\usepackage{array}
\usepackage{comment}
\usepackage{paralist}
\usepackage{xcolor}

\newtheorem{theorem}{Theorem}

\newtheorem{proposition}[theorem]{Proposition}
\newtheorem{conjecture}[theorem]{Conjecture}
\newtheorem{question}[theorem]{Question}

\theoremstyle{definition}

\newtheorem{example}[theorem]{Example}

\newcommand\sO{{\mathcal O}}

\def\geq{\geqslant}

\def\le{\leqslant}
\def\ge{\geqslant}

\newcommand{\PP}{\ensuremath{\mathbb{P}}} 
 
\usepackage{eurosym}

\newcommand{\holom}[3]{\ensuremath{#1:#2  \rightarrow #3}}

\begin{document}
 
\title[A question on effective strictly nef divisors (with an appendix by Andreas H\"oring)]{A question on effective strictly nef divisors \\ (with an appendix by Andreas H\"oring)}
 
\author{Claudio Fontanari}

\address{Claudio Fontanari, Dipartimento di Matematica, Universit\`a degli Studi di Trento, Via Sommarive 14, 38123 Povo, Trento}
\email{claudio.fontanari@unitn.it}

\address{Andreas H\"oring, Universit\'e C\^ote d'Azur, CNRS, LJAD, France, Institut universitaire de France}
\email{Andreas.Hoering@univ-cotedazur.fr}

 
 
\centerline{}
\begin{abstract} 
We introduce and motivate the following question: Is every effective strictly nef Cartier divisor on a projective variety big? In the appendix, Andreas H\"oring produces a counterexample, thus providing a negative answer.
\end{abstract}

\maketitle 

\section{Introduction} 
Let $X$ be a complex projective variety of dimension $n$. 
A Cartier divisor $D$ on $X$ is called \emph{strictly nef} if it has strictly positive intersection product with every curve on $X$. 
Every ample divisor is indeed strictly nef, but after the classical examples by Mumford and Ramanujam 
(see \cite{H}, Chapter I., Examples 10.6 and 10.8) it is well known that the converse does not hold. On the other hand, 
a deep conjecture by Serrano predicts that every strictly nef divisor on a projective manifold becomes ample after a suitable deformation 
in the direction of the canonical divisor $K_X$:

\begin{conjecture} \label{serrano} \emph{(\cite{S})}
If $D$ is a strictly nef divisor on a projective manifold $X$ then $K_X + tD$ is ample for every $t > n+1$. 
\end{conjecture}

Serrano's Conjecture \ref{serrano} holds for surfaces (see \cite{S}), for threefolds with the unique possible exception of Calabi-Yau's
with $D.c_2=0$ (see \cite{S} and \cite{CCP}), for K-trivial fourfolds (see \cite{LM}), and for projective manifolds of 
Kodaira dimension at least $n-2$ (see \cite{CCP}). Otherwise, Conjecture \ref{serrano} is still widely open. 

A weaker version, involving only effective strictly nef divisors, was independently formulated by Beltrametti and Sommese in \cite{BS}, p. 15: 

\begin{conjecture} \label{bs} \emph{(\cite{BS})}
Let $D$ be an effective strictly nef divisor on a projective manifold $X$. If $D-K_X$ is nef then $D$ is ample. 
\end{conjecture}

On the other hand, if the strictly nef divisor $D$ is also big, then Conjecture \ref{serrano} holds for $D$, just by applying \cite{S}, Lemma 1.3: 

\begin{proposition}
If $D$ is a big strictly nef divisor on a projective manifold $X$ then $K_X + tD$ is ample for every $t > n+1$. 
\end{proposition}

Furthermore, if $D - K_X$ is nef then from the ampleness of $K_X + tD$ it follows that also $D$ is ample, 
hence Conjecture \ref{bs} holds for big strictly nef divisors as well:

\begin{proposition}
If $D$ is a big strictly nef divisor on a projective manifold $X$ and $D-K_X$ is nef then $D$ 
is ample. 
\end{proposition}

Finally, also the singular version of Conjecture \ref{serrano} (see \cite{CCP}, 
Conjecture 1.3, and \cite{LOWYZ}, Question 1.4) holds for big strictly nef $\mathbb{Q}$-Cartier divisors. 
Namely, by applying \cite{LOWYZ}, Lemma 5.2 and Lemma 5.3, we deduce: 

\begin{proposition}
If $(X,\Delta)$ is a projective klt pair of dimension $n$ and $D$ is a big strictly nef $\mathbb{Q}$-Cartier divisor on $X$ then $K_X + \Delta + tD$ is ample for every $t >> 0$. 
\end{proposition}

From this point of view, it is remarkable that all examples known so far of strictly nef divisors (see in particular \cite{BS}) have either negative or maximal Iitaka dimension. This experimental observation suggests the following question:

\begin{question} \label{big}
Is every effective strictly nef Cartier divisor on a projective variety big?
\end{question}

Even though this is trivially true for surfaces, in higher dimension the answer turns out to be far less obvious, so it seems wise to adopt a fully agnostic attitude. 

As a starting point, we recall that from the case of surfaces it follows a partial positive result in any dimension (see \cite{C}, Proposition 22): 
if $D$ is a strictly nef Cartier divisor on a projective variety of dimension $n$ with Iitaka dimension $\kappa(D) \ge n-2$ then $D$ is big. 
Hence the first instance to be addressed is the one of a strictly nef divisor $D$ with $\kappa(D) =0$  on a threefold. 
We point out that, once this case were ruled out in arbitrary dimension, then the whole picture would become clear. 
Namely, we formulate the following a priori weaker question: 

\begin{question} \label{zero}
Does every effective strictly nef Cartier divisor $D$ on a projective variety $X$ satisfy $h^0(X,mD) \ge 2$ for some $m \ge 1$?
\end{question}

We show that an affirmative answer to Question \ref{zero} would imply an affirmative answer to Question \ref{big}:

\begin{theorem} \label{main} 
Assume that every effective strictly nef Cartier divisor $D$ on every projective variety $X$ of dimension 
$\dim(X) \le n$ satisfies $h^0(X,mD) \ge 2$ for some $m \ge 1$.
Then every effective strictly nef Cartier divisor on a projective variety of dimension $n$ is big. 
\end{theorem}

We also present an unconditional  result pointing towards the same direction: 

\begin{theorem} \label{affine}
If $D$ is an effective strictly nef Cartier divisor on a projective variety and the schematic base locus of $\vert m D \vert$ becomes constant for large $m$ then $D$ is big.
\end{theorem}

In the opposite direction, by closely following \cite{H}, Chapter I., Example 10.8 and \cite{BS}, Example 1.2, we adapt Ramanujam construction and define an inductive procedure to build a strictly nef divisor which is effective but not big (see Example \ref{ramanujam}). This approach 
works by induction on the dimension of the ambient projective manifold and what is needed 
to obtain a negative answer to Question \ref{big} is indeed the base of the induction. This is provided in Appendix \ref{andreas} by Andreas H\"oring, hence both Question \ref{big} and Question \ref{zero} have a negative answer in any dimension $n \ge 3$. 


\medskip

{\bf Acknowledgements:} 
The author is grateful to Edoardo Ballico, Fr\'ederic Campana, Paolo Cascini, and Thomas Peternell for their helpful remarks. 
The author is a member of GNSAGA of the Istituto Nazionale di Alta Matematica "F. Severi". 

\section{The proofs}

\noindent \textit{Proof of Theorem \ref{main}.} We argue by induction on $n$, the case $n=1$ being obvious. 
If $D$ is an effective strictly nef divisor on a projective variety $V$ of dimension $n$, 
by assumption we have $h^0(V, mD ) \ge 2$ for some $m \ge 1$. 
Let $E = \sum a_i E_i$ be an effective divisor linearly equivalent to $mD$. 
Since $h^0(V, mD ) \ge 2$ there exists $i$ such that $D$ restricts to an effective strictly nef Cartier divisor on $E_i$. 
From the inductive assumption applied to the projective variety $E_i$ of dimension $n-1$ we obtain that the restriction 
of $D$ to $E_i$ is big. On the other hand, if $D$ is not big then 
$$
0 = D^n = mD^n = D^{n-1}.E = \sum a_i D^{n-1}.E_i
$$
Since $D$ is nef and $E$ is effective it follows that $D^{n-1}.E_i = 0$ for every $i$, in particular the restriction 
of $D$ to $E_i$ is not big. This contradiction ends the proof.

\qed

\noindent \textit{Proof of Theorem \ref{affine}.} The complement $U$ of the support of 
the effective strictly nef Cartier divisor $D$ does not contain complete curves, hence from
\cite{H}, Chapter II., Theorem 5.1 and the Remarks following its statement, we deduce that $U$ is affine. Now the claim is a direct consequence of Goodman's criterion (\cite{H}, Chapter II., Theorem 6.1), see for instance the 
statement of Theorem 3.1 in \cite{Z}: indeed, it follows easily from the remarks after the statement of Theorem 2.1 on p. 803 and at the beginning of the proof of Theorem 3.1 on p. 808.

\qed

\begin{example}\label{ramanujam}
Let $Y$ be a projective manifold of dimension $n-1 \ge 3$ with an effective strictly nef divisor $D$ such that $D^{n-1} =0$.

Define $X := \mathbb{P}(\mathcal{O}_Y(D) \oplus \mathcal{O}_Y)$. If $X_0$ is the zero section of the projection $\pi: X \to Y$, then $X_0$ corresponds to the tautological line bundle on $X$ and by arguing as in \cite{BS}, Example 1.2, we compute 
\begin{equation}\label{one}
X_0^i = D^{i-1}_{X_0}
\end{equation}
for every $i$ with $2 \le i \le n$. 

Define $E := X_0 + \pi^*D$. From (\ref{one}) with $i=2$ we deduce 
\begin{equation}\label{two}
E_{X_0} = (X_0 + \pi^*D)_{X_0} = 2 D_{X_0}.
\end{equation}

Then the following holds: 

(i) $E$ is effective because $D$ is effective. 

(ii) $E$ is strictly nef. Indeed, let $C$ be a curve on $X$. 

If $C$ is a fibre of $\pi$, then by the projection formula we have $E.C = X_0.C + \pi^*D.C = X_0.C + D.\pi_*C = 1 + 0 > 0$. 

If $C \subset X_0$, then by (\ref{two}) we have $E.C = E_{X_0}.C_{X_0} = 2 (D.C)_{X_0} > 0$.

If $C \not \subset X_0$ and $\pi(C)$ is a curve in $Y$, then by the projection formula we have 
$E.C = X_0.C + \pi^*D.C = X_0.C + D.\pi_*C > 0$ since $X_0.C \ge 0$ and $D.\pi_*C > 0$.

(iii) $E$ is not big. Indeed, we have $E^n = (X_0 + \pi^*D)^n = 0$
because $\pi^*D^n=0$, $X_0.(\pi^*D)^{n-1} = D^{n-1}_{X_0} = 0$, and 
$X_0^i. (\pi^*D)^{n-i} = 0$ for every $i$ with $2 \le i \le n$ by (\ref{one}).
\end{example}

\section{Appendix by Andreas H\"oring}\label{andreas}

Let $\holom{\pi}{\PP(V)}{C}$ be a Mumford example, i.e.
$C$ is a smooth projective curve of genus $g>1$ and 
$V$ is a rank two vector bundle on $C$ such that $c_1(V)=0$
and
$$
H^0(C, S^m V) = 0 \qquad \forall \ m \in \mathbb{N}.
$$
Then it is known \cite[Ex.1.5.2]{L} that the tautological class $c_1(\sO_{\PP(V)}(1))$ is strictly nef, but not big.
Observe that by Serre duality and Riemann-Roch
$$
h^1(C, V^*) = h^0(C, K_C \otimes V) \geq \chi(C, K_C \otimes V) = c_1(K_C \otimes V) + 2 \chi(\sO_C) = 2 (g-1)>0,
$$
so there exists an extension
$$
0 \rightarrow \sO_C \rightarrow W \rightarrow V \rightarrow 0
$$
such that the extension class
$\eta \in H^1(C, \sO_C \otimes V^*)$
is not zero. Let
$
\holom{p}{X:=\PP(W)}{C}$
be the projectivisation, and observe that it contains
$$
Y := \PP(V) \subset \PP(W) = X
$$
as a prime divisor such that $[Y]=c_1(\sO_{\PP(W)}(1))$.

\begin{proposition}
The prime divisor $Y \subset X$ is effective, strictly nef, but not big.
\end{proposition}

\begin{proof}
Let $C \subset X$ be an irreducible curve. If $C \subset Y$ then
$$
Y \cdot C = c_1(\sO_{\PP(W)}(1))|_Y \cdot C =
c_1(\sO_{\PP(V)}(1)) \cdot C > 0
$$
since $c_1(\sO_{\PP(V)}(1))$ is strictly nef. Thus $Y$ is nef and strictly nef on all curves contained in its support.
Clearly $Y$ is not big, since
$$
Y^3 = c_1(\sO_{\PP(W)}(1))|_Y^2 = c_1(\sO_{\PP(V)}(1))^2=0.
$$
Assume now that there exists an irreducible curve such that $Y \cdot C=0$. Since $C \not\subset Y$ this implies that $C \subset (X \setminus Y)$.
Yet $\eta \neq 0$, so by \cite[Lemma 3.9]{HP} the complex manifold
$X \setminus Y$ contains no positive-dimensional compact subvarieties of positive dimension, a contradiction.

\end{proof}

\end{document}